\documentclass[11pt]{amsart}
\usepackage{amssymb}
\usepackage{amsmath}
\usepackage{amsthm}
\usepackage{verbatim}
\usepackage[all]{xy}
\usepackage{url}
\usepackage{mathtools}
\usepackage{dsfont}

\numberwithin{equation}{section}

\theoremstyle{plain}
\newtheorem{theorem}[equation]{Theorem}

\newtheorem{proposition}[equation]{Proposition}
\newtheorem{lemma}[equation]{Lemma}
\newtheorem{corollary}[equation]{Corollary}

\theoremstyle{definition}

\newtheorem{examples}[equation]{Examples}
\newtheorem{remark}[equation]{Remark}
\newtheorem{remarks}[equation]{Remarks}
\newtheorem{definition}[equation]{Definition}
\newtheorem{construction}[equation]{Construction}

\newtheorem{subsec}[equation]{}

\newcommand{\C}{{\mathds C}}
\newcommand{\R}{{\mathds R}}

\newcommand{\Z}{{\mathds Z}}
\newcommand{\G}{{\mathbb G}}

\newcommand{\BBB}{{\sf B}}

\newcommand{\DDD}{{\sf D}}
\newcommand{\EEE}{{\sf E}}
\newcommand{\FFF}{{\sf F}}

\newcommand{\cZ}{{\mathcal{Z}}}
\newcommand{\cN}{{\mathcal{N}}}

\newcommand{\into}{\hookrightarrow}

\newcommand{\ov}{\overline}

\DeclareMathOperator{\Aut}{Aut}

\DeclareMathOperator{\Cl}{Cl}

\DeclareMathOperator{\Gal}{Gal}

\DeclareMathOperator{\Hom}{Hom}

\newcommand{\ad}{{\rm ad}}

\newcommand{\Ad}{{\rm Ad}}

\newcommand{\sss}{{\rm ss}}

\newcommand{\isoto}{\overset{\sim}{\to}}
\newcommand{\onto}{\twoheadrightarrow}
\newcommand{\labelto}[1]{\xrightarrow{\makebox[1.5em]{\scriptsize ${#1}$}}}

\newcommand{\hs}{\kern 0.8pt}
\newcommand{\hssh}{\kern 1.2pt}
\newcommand{\hshs}{\kern 1.6pt}
\newcommand{\hssss}{\kern 2.0pt}

\newcommand{\hm}{\kern -0.8pt}
\newcommand{\hmm}{\kern -1.2pt}

\renewcommand{\hbar}{{\bar h}}

\newcommand{\X}{{{\sf X}}}

\newcommand{\id}{{\rm id}}

\newcommand{\sZ}{{\mathcal Z}}

\newcommand{\Tt}{{\mathcal T}}

\newcommand{\kbar}{{\bar k}}

\newcommand{\red}{{\rm red}}

\begin{document}

\title[Quasi-connected reductive groups]%
{Galois cohomology of real quasi-connected\\ reductive groups}

\author{Mikhail Borovoi}
\address{%
Borovoi: Raymond and Beverly Sackler School of Mathematical Sciences,
Tel Aviv University,
6997801 Tel Aviv,
Israel}
\email{borovoi@tauex.tau.ac.il}

\author{Andrei A. Gornitskii}
\address{%
Gornitskii: Moscow Center of Fundamental and Applied Mathematics,
119991 Moscow,
Russian Federation,
and
Raymond and Beverly Sackler School of Mathematical Sciences,
Tel Aviv University,
6997801 Tel Aviv,
Israel}
\email{gnomage@mail.ru}

\author{Zev Rosengarten}
\address{%
Rosengarten: Einstein Institute of Mathematics,
The Hebrew University of Jerusalem,
Edmond J. Safra Campus, 91904, Jerusalem, Israel}
\email{zev.rosengarten@mail.huji.ac.il}

\thanks{Borovoi and Gornitskii were  supported
by the Israel Science Foundation (grant 870/16). Gornitskii
was supported by the Ministry of Education and Science of
the Russian Federation as part of the program of the Moscow Center for
Fundamental and Applied Mathematics under the agreement
No. 075-15-2019-1621.
Rosengarten was supported by a Zuckerman Postdoctoral Scholarship.}



\date{\today}

\begin{abstract}
By a quasi-connected reductive group (a term of Labesse)
over an arbitrary field we mean an almost direct product
of a connected semisimple group and a quasi-torus
(a smooth group of multiplicative type).
We show that a linear algebraic group is quasi-connected reductive
if and only if it is isomorphic to a smooth normal subgroup
of a connected reductive group.
We compute the first Galois cohomology set  $H^1(\R,G)$
of a quasi-connected reductive group $G$
over the field $\R$ of real numbers
in terms of a certain action of a subgroup of the Weyl group
on the Galois cohomology of a fundamental quasi-torus of $G$.
\end{abstract}

\maketitle

\section{Introduction}
We denote by $\R$ and $\C$ the fields of real numbers
and of complex numbers, respectively.
Let $G$ be a linear algebraic group over  $\R$
(for brevity we say that $G$ is an $\R$-group).
  We are interested in the first Galois cohomology set
$H^1(\R,G)\coloneqq H^1\big(\text{Gal}(\C/\R), G(\C)\big)$.
In terms of Galois cohomology one can state
answers to many natural questions;
see Serre  \cite[Section III.1]{S} and  Berhuy \cite{Be}.

The Galois cohomology of {\em compact} $\R$-groups was computed
by Borel and Serre \cite[Theorem 6.8 and Example (a), p. 157]{BS},
see also Serre's book \cite[Section III.4.5, Theorem 6 and Example (a)]{S}.
The Galois cohomology of {\em connected}  reductive $\R$-groups
was computed by the first-named author \cite{Bo},
see also \cite{Borovoi-arXiv}.
This result was used by Borovoi and Evenor \cite{BE}
in order to compute explicitly the Galois cohomology
(namely, representatives of all cohomology classes)
for all simply connected absolutely simple $\R$-groups.
Recently this result was used by Borovoi and Timashev \cite{BT}
in order to compute explicitly the Galois cohomology of all semisimple groups,
and by Nair and Prasad in their article \cite{NP}
on cohomological representations of real reductive groups.

In Section \ref{s:arbitrary} of this article
we consider a class of not necessarily connected smooth reductive groups
over an arbitrary field $k$, which following Labesse \cite{La}
we call {\em quasi-connected reductive} $k$-groups.
We say that a $k$-group is quasi-connected reductive if
it is isomorphic to an almost direct product
of a connected semisimple $k$-group and a $k$-quasi-torus
(a smooth $k$-group of multiplicative type).
We show that a $k$-group is quasi-connected reductive
if and only if it is isomorphic to a smooth normal subgroup
of a connected reductive $k$-group;
see Theorems \ref{t:normal-sub} and \ref{t:quasi-normal}.

Many known results on connected reductive groups and their Galois cohomology
can be extended to quasi-connected reductive groups.
In particular, generalizing \cite[Corollary 5.4.1]{Bo-Memoir},
Labesse computed the Galois cohomology
of quasi-connected reductive groups  over $p$-adic fields;
see \cite[Proposition 1.6.7]{La}.
In this article, generalizing \cite[Theorem 1]{Bo}, we prove Theorem \ref{thm:bijectivity}
that computes $H^1(\R,G)$ for a quasi-connected reductive group $G$  over $\R$.

\section{Quasi-connected reductive groups over an arbitrary field}
\label{s:arbitrary}

\begin{subsec}{\bf Notation and conventions.}
In this article $k$ is a field,
$\kbar$ is a fixed algebraic closure of $k$,
and $k_s$ is the separable closure of $k$ in $\kbar$.
We denote by $\Gal(k_s/k)$ the Galois group of $k_s$ over $k$.

By an {\em algebraic $k$-group} (for brevity: a $k$-group)
we mean an {\em affine} group scheme of finite type over $k$, not necessarily smooth.
Let $G$ be a $k$-group.
We write $G_\kbar$ for the base change $G\times_k \kbar$.
We denote by $G(k)$ and $G(\kbar)$ the groups of $k$-points
and of $\kbar$-points of $G$, respectively.
In this section we write $|G|$ for $G(\kbar)$
(because the formula $|Z(G)|$ is easier to read than $Z(G)(\kbar)$\hs).

We say that a $k$-group $H$ is {\em reduced} if it is reduced as a scheme.
Since $H$ is affine, it is reduced if and only if
the corresponding Hopf algebra $k[H]$ has no nontrivial nilpotents.
We mostly work with {\em smooth} $k$-groups.
Note that a $k$-group $H$ is smooth if and only if it is {\em geometrically reduced},
that is, $H_\kbar$ is reduced; see Milne \cite[Proposition 1.26(b)]{Milne-AG}.
By Cartier's theorem, in characteristic 0 all $k$-groups are smooth;
see \cite[Theorem 3.23]{Milne-AG}.
If $k$ is perfect, then a $k$-group is smooth if and only if it is reduced;
see \cite[Proposition 1.26(a)]{Milne-AG}.
Let $H_\red\subseteq H$ denote the reduced subscheme
corresponding to the quotient of $k[H]$ by its nilradical.
We have $H_\red(\kbar)=H(\kbar)$.
If $k$ is perfect, then $H_\red$ is a $k$-subgroup of $H$;
see \cite[Corollary 1.39]{Milne-AG}.
However, over a nonperfect field $k$,
the $k$-subscheme $H_\red$ might not be a $k$-subgroup;
see  \cite[Example 1.57]{Milne-AG}.

Let $G$ be a smooth algebraic $k$-group, not necessarily connected.
We write $G^0$ for the identity component of $G$
and $\pi_0(G)$ for the finite \`etale $k$-group $G/G^0$.
We denote by $R_u(G_\kbar)$ the unipotent radical of $G^0_\kbar$,
that is, the largest smooth connected normal unipotent subgroup of $G^0_\kbar$\hs.
Note that in characteristic 0 all unipotent groups are smooth and connected.
We say that $G$ is {\em reductive} if it is smooth and $R_u(G_\kbar)=\{1\}$.
If $G$ is reductive,  we write $G^\sss$  for the commutator subgroup
$( G^0,\,G^0)$ of the connected reductive $k$-group  $G^0$;
then $G^\sss$ is a connected semisimple $k$-group.
We write $Z(G)$ for the center of $G$; it might be non-smooth.
If $k$ is perfect, then $Z(G)_\red$ is a smooth $k$-subgroup
of $Z(G)$ satisfying $Z(G)_\red(\kbar)=Z(G)(\kbar)$.
\end{subsec}

\begin{subsec}\label{ss:quasi-torus}
Following Gorbatsevich, Onishchik, and Vinberg  \cite[Section 3.3.2]{GOV},
we say that a {\em quasi-torus} over $k$ is a smooth $k$-group of multiplicative type.
An equivalent definition of a quasi-torus is a
``smooth commutative $k$-group consisting of semisimple elements''
(that is, all elements of $Q(\kbar)$ are semisimple);
see \cite[Corollary 12.21]{Milne-AG}.
Then any $k$-torus is a $k$-quasi-torus,
and any smooth $k$-subgroup of a $k$-quasi-torus is a $k$-quasi-torus.
Conversely, any $k$-quasi-torus $Q$ admits an embedding into a $k$-torus.
Indeed, the character group $\X^*(Q_\kbar)$ is a finitely generated $\Gal(k_s/k)$-module
and hence it is a quotient module of a $\Z$-free $\Gal(k_s/k)$-module.
See \cite[Theorem 12.23]{Milne-AG}.
\end{subsec}

\begin{definition}\label{d:qc}
An algebraic group $G$ (not necessarily connected) over a  field $k$
is called {\em quasi-connected reductive} if
\begin{enumerate}
\item[(1)] $G$ is reductive;
\item[(2)] $|G|=|G^\sss|\cdot |Z(G)|$;
\item[(3)] $|Z(G)|$ consists of semisimple elements,
\end{enumerate}
where we write $|G|$ for $G(\kbar)$ and $|Z(G)|$ for $Z(G)(\kbar)$.
\end{definition}

This  definition is equivalent
to the definition of Labesse \cite[Definition 1.3.1]{La};
see Remark \ref{r:Labesse} below.
Note that in Section \ref{s:cohomology} we shall use Definition \ref{d:qc}
rather than the definition of Labesse.

\begin{examples}\label{ex:i-ii-iii}\
\begin{enumerate}
\item[(i)] Any connected reductive  $k$-group is  quasi-connected reductive.
\item[(ii)] Any $k$-quasi-torus is quasi-connected reductive.
\item[(iii)] The product of two quasi-connected reductive $k$-groups
is quasi-connect\-ed reductive.
\end{enumerate}
\end{examples}

\begin{lemma}
A $k$-group $G$ satisfying conditions $(1)$ and $(2)$
of Definition \ref{d:qc} satisfies condition $(3)$
if and only if it satisfies the following condition:
\begin{enumerate}
\item[($3'$)]  $|\pi_0(G)|$ consists of semisimple elements.
\end{enumerate}
\end{lemma}

\begin{proof}
By (2) the homomorphism $|Z(G)| \to |\pi_0(G)|$ is surjective.
If (3) holds, then $|Z(G)|$ consists of semisimple elements,
and hence the same is true for $|\pi_0(G)|$, which gives $(3')$.

Conversely, we have a short exact sequence
\begin{equation}\label{e:cap-G0}
 1\to \,|Z(G)| \cap |G^0|\,\to \,|Z(G)| \,\to\,
       |Z(G)| \hs/\hs\big(\hs|Z(G)|\cap |G^0|\hs\big)\to 1.
\end{equation}
By (1) $G^0$ is a connected reductive $k$-group,
and hence $Z(G^0)$ a $k$-group of multiplicative type (not necessarily smooth);
see  \cite[Proposition 21.7]{Milne-AG}.
It follows that $|Z(G^0)|$ consists of semisimple elements
(see also Humphreys \cite[\hs Corollary 26.2A(b)\hs]{Hu}\hs).
The left-hand group in \eqref{e:cap-G0} is a subgroup of $|Z(G^0)|$
and hence consists of semisimple elements.
The right-hand group is a subgroup of $|\pi_0(G)|$. If $(3')$ holds,
then $|\pi_0(G)|$ and hence the right-hand group consist of semisimple elements.
We conclude that $|Z(G)|$ consists of semisimple elements, that is, (3) holds.
\end{proof}

\begin{remarks}
\begin{enumerate}
\item[(i)] In characteristic 0, any finite algebraic group consists of semisimple elements,
which gives $(3')$ and hence (3).
It follows that an algebraic group over a field $k$ of characteristic 0
is quasi-connected reductive if and only if
conditions (1) and (2) of Definition \ref{d:qc} are satisfied.
\item[(ii)]  In characteristic $p>0$, the group $|\pi_0(G)|$ consists of semisimple elements
if and only if its order is prime to $p$.
\end{enumerate}
\end{remarks}

\begin{proposition}\label{p:G-sur-G'}
Let $\varphi\colon G\to G'$ be a surjective (on $\kbar$-points)
 homomorphism of smooth $k$-groups.
If $G$ is quasi-connected reductive, then so is $G'$.
\end{proposition}

\begin{proof}
We have $\varphi(G^0)=G^{\prime\hs0}$
and hence $G^{\prime\hs0}$ is reductive, which gives (1).
We have $\varphi(G^\sss)=G^{\prime\hs\sss}$.
Since $\varphi$ is surjective, we have $\varphi|Z(G)|\subseteq |Z(G')|$,
where we write $\varphi|Z(G)|$ for $\varphi(|Z(G)|)$.
Thus
\[ |G'|=\varphi|G|=\varphi\hs|G^\sss\hm\cdot\hm Z(G)|=\varphi|G^\sss|\cdot\varphi|Z(G)|
       \subseteq |G^{\prime\hs\sss}|\cdot |Z(G')|,\]
which gives (2).  The homomorphism $\varphi_*\colon |\pi_0(G)|\to|\pi_0(G')|$
is surjective, and hence $|\pi_0(G')|$ consists
of semisimple elements, which gives $(3')$.
\end{proof}

\begin{proposition}
\label{p:zev-center}
Let $G$ be a quasi-connected reductive $k$-group.
Then $Z(G)$ is a $k$-group of multiplicative type.
\end{proposition}

\begin{proof}
We have a short exact sequence of $k$-groups
\begin{equation}\label{e:Z(G)}
1\to\hs Z(G)\cap G^0\to Z(G)\hs\to\hs Z(G)\hs/\hs\big(Z(G)\cap G^0\big)\hs\to 1.
\end{equation}
Here $Z(G)\cap G^0$ is a central $k$-subgroup of the connected reductive $k$-group $G^0$
and hence  a $k$-group of multiplicative type (not necessarily smooth);
see  \cite[Proposition 21.7]{Milne-AG}.
On the other hand, by (3$'$)  the \`etale $k$-group $\pi_0(G)$ consists of semisimple elements
(that is, $|\pi_0(G)|$ consists of semisimple elements).
It follows that the $k$-subgroup
\[Z(G)/\big(Z(G)\cap G^0\big)\hs\subseteq\hs \pi_0(G)\]
is a commutative \`etale $k$-group consisting of semisimple elements,
and hence a $k$-group of multiplicative type.
Since in \eqref{e:Z(G)} the right-hand group and the left-hand group are of multiplicative type,
we conclude that the commutative $k$-group $Z(G)$
is of multiplicative type; see \cite[Corollary 12.22]{Milne-AG}.
\end{proof}

\begin{proposition}
\label{p:zev-mult}
Let $M$ be a $k$-group of multiplicative type.
Then the $k$-subscheme $M_\red\subseteq M$ is a smooth $k$-subgroup
of multiplicative type, that is, a $k$-quasi-torus.
\end{proposition}

\begin{proof}
We may and shall assume that $k$ is a field of positive characteristic $p$.
Let $\X^*(M)$ denote the character group of $M$
regarded as a $\Gal(k_s/k)$-module.
The subgroup $\X^*(M)(p) \subset \X^*(M)$ of $p$-power torsion elements
is a finite $\Gal(k_s/k)$-invariant subgroup.
Set $Y= \X^*(M)/\hs\X^*(M)(p)$.
Then we have a short exact sequence of $\Gal(k_s/k)$-modules
\[0\to \X^*(M)(p)\to \X^*(M)\to Y\to 0\]
and the dual short exact sequence of $k$-groups of multiplicative type
\[1\to Q\to M\to M/Q\to 1,\]
where $Q$ is the $k$-group of multiplicative type with character group $Y$.
In the latter exact sequence, $Q$ is smooth, because $Y$ has no $p$-power torsion.
Moreover, $M/Q$ is an infinitesimal group (it has no nontrivial $\kbar$-points),
because its character group  $\X^*(M)(p)$ is a $p$-power torsion group.
We conclude that $M_\red=Q$.
Therefore, $M_\red$ is a smooth $k$-subgroup of multiplicative type of $M$,
and hence a $k$-quasi-torus.
\end{proof}

\begin{corollary}
\label{c:zev-subgroup}
Let $G$ be a quasi-connected reductive $k$-group.
Then the $k$-subscheme $Z(G)_\red\subseteq Z(G)$ is a smooth $k$-subgroup
of multiplicative type, that is, a $k$-quasi-torus.
\end{corollary}

\begin{proof}
The corollary follows from Propositions \ref{p:zev-center} and \ref{p:zev-mult}.
\end{proof}

\begin{theorem}\label{c:quotient}
A $k$-group is quasi-connected reductive if and only if
it is isomorphic to the quotient of the direct product
of a connected reductive $k$-group and a $k$-quasi-torus
by a finite central $k$-subgroup (not necessarily smooth).
\end{theorem}

\begin{proof}
By Examples \ref{ex:i-ii-iii} and Proposition \ref{p:G-sur-G'},
the quotient of the direct product
of a connected reductive $k$-group and a $k$-quasi-torus
by any normal $k$-subgroup is quasi-connected reductive
(here the quotient is smooth by \cite[Corollary 5.26]{Milne-AG}).

Conversely, assume that $G$ is quasi-connected reductive.
By Corollary \ref{c:zev-subgroup}, the subscheme $Q\coloneqq Z(G)_\red$ of $Z(G)$ is a $k$-quasi-torus.
Consider the homomorphism
\[\varphi\colon G^\sss\times Q\to G,\quad (s,q)\mapsto sq^{-1}\ \
      \text{for}\ s\in G^\sss,\, q\in Q.\]
We have $|Q|=|Z(G)|$.
By property (2), the homomorphism $\varphi$ is surjective on $\kbar$-points,
and hence it is surjective on closed points.
Since $G$ is smooth, $\varphi$ is faithfully flat; see \cite[Proposition 1.70]{Milne-AG}.
It follows that $G$ is isomorphic to $(G^\sss\times Q)/\ker\varphi$;
see \cite[Definition 5.5]{Milne-AG}.
Here the $k$-group $G^\sss$ is connected semisimple, and hence connected reductive,
and $\ker\varphi\cong G^\sss\cap Q$ is a finite central subgroup of $G^\sss\times Q$.
\end{proof}

\begin{theorem}\label{t:normal-sub}
Any smooth normal $k$-subgroup $G$ of a quasi-connected reductive $k$-group $H$
is quasi-connected reductive.
\end{theorem}

\begin{proof}
We may and shall assume that $k$ is algebraically closed.
We check conditions (1--3) of Definition \ref{d:qc}.
The unipotent radical $R_u(G^0)$ is a characteristic subgroup of $G$,
that is, it is preserved by all $k$-automorphisms of $G$.
It follows that  $R_u(G^0)$ is normalized by $H(k)$,
and hence by $H$, because $k$ is algebraically closed and   $H$ is smooth.
Thus $R_u(G^0)$ it is a smooth, connected, unipotent, normal subgroup of $H$.
Since $H$ is reductive, we conclude that  $R_u(G^0)=\{1\}$.
It follows that $G$ is reductive as well, which gives (1).

Consider the homomorphism $\Ad_H\colon H\to H^\ad\coloneqq H/Z(H)$;
see \cite[Definition 5.5 and Proposition 5.14]{Milne-AG}
for the definition of a quotient by a not necessarily smooth normal subgroup.
Consider the $k$-subgroup $Z(H^\sss)=\sZ_{H^\sss}\hs H^\sss$ of $H^\sss$,
where $\sZ_{H^\sss}$ denotes the schematic centralizer in $H^\sss$.
Note that $\sZ_{H^\sss}\hs |Z(H)|=H^\sss$, where we write $|Z(H)|$ for $Z(H)(k)$.
We have
\begin{align*}
Z(H^\sss)=\sZ_{H^\sss}\hs H^\sss= &\sZ_{H^\sss}\hs |H^\sss|
                 =\sZ_{H^\sss}\hs|H^\sss|\cap\sZ_{H^\sss}\hs |Z(H)|\\
=&\sZ_{H^\sss}\hs \big(|H^\sss|\cdot |Z(H)|\big) =\sZ_{H^\sss}\hs |H|=\sZ_{H^\sss}\hs H,
\end{align*}
where the equalities $\sZ_{H^\sss}\hs H^\sss= \sZ_{H^\sss}\hs |H^\sss|$
and $\sZ_{H^\sss}\hs |H|=\sZ_{H^\sss}\hs H$
follow from the fact that the $k$-groups $H^\sss$ and $H$ are smooth, and  the equality
$\sZ_{H^\sss}\big(|H^\sss|\cdot |Z(H)|\big) =\sZ_{H^\sss}\hs |H|$
follows from property (2) for $H$ in Definition \ref{d:qc}.
Thus
\[Z(H^\sss)=\sZ_{H^\sss} H\hs\subseteq \sZ_H\hs H=Z(H),\]
whence $H^\ad= H^\sss/Z(H^\sss)$.
It follows that  $H^\ad$ is a semisimple group of adjoint type,
and hence a product of simple algebraic groups of adjoint type.

The group of $k$-points of any simple algebraic group of adjoint type
over an algebraically closed field $k$
is noncommutative and simple as an abstract group;
see \cite[Corollary 29.5]{Hu}.
It follows that all normal algebraic subgroups of $H^\ad$
are connected and all non-unit normal subgroups of $|H^\ad|$ are non-commutative.
Since $\Ad_H\hs |Z(G^0)|$ is a commutative normal subgroup
of  $|H^\ad|$, it is trivial.
Thus $|Z(G^0)|\subseteq |Z(H)|$, whence
\begin{equation}\label{e:ZG0-ZG}
|Z(G^0)|\subseteq |G|\cap |Z(H)|\subseteq |Z(G)|.
\end{equation}
Since  $\Ad_H\hs G$ is a normal subgroup of $H^\ad$, it is connected.
It follows that $\Ad_H\hs G=\Ad_H\hs G^0$, and hence
\begin{equation}\label{e:G0-Zr}
|G|= |G^0|\cdot \big(\hs |G|\hs\cap\hs |Z(H)|\hs\big)=|G^0|\cdot |Z(G)| .
\end{equation}
Since $G^0$ is connected and reductive, we have
\begin{equation}\label{e:Gss-Z(G0)r}
|G^0|=|G^\sss|\cdot |Z(G^0)|.
\end{equation}
From  \eqref{e:G0-Zr}, \eqref{e:Gss-Z(G0)r}, and \eqref{e:ZG0-ZG}, we obtain
\[|G|=|G^0|\cdot  |Z(G)|  =|G^\sss|\cdot |Z(G^0)|\cdot |Z(G)|  =|G^\sss|\cdot |Z(G)|, \]
which gives (2).

Since $\Ad_H\hs |Z(G)|$ is a commutative normal subgroup of  $|H^\ad|$, it is trivial.
Thus $|Z(G)|\subseteq |Z(H)|$.
Since $|Z(H)|$ consists of semisimple elements,
the same holds for the group $|Z(G)|$ as well, which gives (3).
We conclude that $G$ is quasi-connected reductive.
\end{proof}

\begin{theorem}\label{t:quasi-normal}
Any quasi-connected  reductive $k$-group $G$
is isomorphic to the kernel
of a surjective $k$-homomorphism  $H\onto T$
of a connected reductive $k$-group $H$ onto a $k$-torus $T$, and hence $G$
is isomorphic to a smooth normal $k$-subgroup of a connected reductive $k$-group.
\end{theorem}

\begin{proof}
By Corollary \ref{c:zev-subgroup},
the $k$-subscheme $Q\coloneqq Z(G)_\red$ of $Z(G)$
is a $k$-quasi-torus.
We choose an embedding $\iota\colon Q\into T'$,
defined over $k$, into a $k$-torus $T'$,
and consider the short exact sequence of $k$-groups
\begin{equation}\label{e:id-iota}
1\to G^\sss\times Q \labelto\alpha  G^\sss\times T'
      \labelto\beta T'/Q  \to 1,
\end{equation}
where $\alpha=\id\!\times\!\iota$ and
$\beta(s,t)=t\cdot Q$ for $s\in G^\sss,\, t\in T'$.
We identify $Q$ with its image in $T'$, and we identify
$G^\sss\times Q$ with its image in $G^\sss\times T'$.
We consider the homomorphism of smooth $k$-groups
$$\varphi\colon G^\sss\times Q  \to G,\quad (s,q)\mapsto s\cdot q^{-1}\ \,
      \text{ for }\, s\in G^\sss,\ q\in Q=Z(G)_\red . $$
Then \hs$\ker\varphi=\big\{\hs(q,q)\hs\mid\hs q\in G^\sss\cap Q \hs\big\}$
\hs is clearly central in $G^\sss\times T'$ (but not necessarily smooth).
By taking quotients by the central $k$-subgroup $\ker\varphi$,
we obtain from \eqref{e:id-iota} a short exact sequence of $k$-groups
\[
1\to \big(G^\sss\times Q \big)/\ker\varphi\labelto{\bar\alpha}
    (G^\sss\times T')/\ker\varphi \labelto{\bar\beta} T'/Q  \to 1;
\]
see \cite[Proposition 5.18]{Milne-AG}.
We see  that $\big( G^\sss\times Q \big)/\ker\varphi$
is the kernel of the surjective (on $\kbar$-points) $k$-homomorphism
\begin{align*}
&\bar\beta\colon (G^\sss\times T')/\ker\varphi \to T'/Q
\end{align*}
from the connected reductive $k$-group $H\coloneqq(G^\sss\times T')/\ker\varphi$
onto the $k$-torus $T\coloneqq T'/Q$.
It remains to observe that by the proof of Corollary \ref{c:quotient}, the quotient
$(G^\sss\times Q)/\ker\varphi$ is isomorphic to $G$.
\end{proof}

\begin{remark}\label{r:Labesse}
Labesse \cite[Definition 1.3.1]{La} defines a quasi-connected reductive group
as the kernel of a surjective homomorphism
of a connected reductive group onto a torus.
By Theorems \ref{t:normal-sub} and \ref{t:quasi-normal},
a reductive $k$-group is quasi-connected in the sense of our Definition \ref{d:qc}
if and only if it is quasi-connected in the sense of Labesse and smooth.
\end{remark}

\section{Galois cohomology over $\R$}
\label{s:cohomology}

\begin{subsec}
From now on $k=\R$ and $\kbar=\C$.
Let $G$ be a an algebraic $\R$-group.
By abuse of notation, in this section we identify $G$ with $G(\C)$
(and do not write $|G|$ for $G(\kbar)=G(\C)$, as in Section \ref{s:arbitrary}).
In particular, $g\in G$ means that $g\in G(\C)$.
We recall the definition of the first
Galois cohomology set $H^1(\R,G)$.
  The set of 1-cocycles is
$$Z^1(\R,G)=\{c\in G\mid c\hs\bar{c}=1\},$$
where the bar denotes  complex conjugation.
  The group $G(\C)$ acts on the right on $Z^1(\R,G)$ by
$c*x=x^{-1}c\hs\bar{x}$,  where $c\in Z^1(\R,G)$ and $x\in G(\C)$.
By definition $H^1(\R,G)= Z^1(\R,G)/G(\C)$.
\end{subsec}

\begin{subsec}
The Galois cohomology of an $\R$-quasi-torus $Q$ is known.
Namely, write $\X^*(Q)=\Hom_\C(Q,\G_{m,\R})$
for the character group of $Q$,
where $\G_{m,\R}$ denotes the multiplicative group over $\R$.
Then $\X^*(Q)$ is a finitely generated abelian group.
The Galois group $\Gal(\C/\R)$ naturally acts on $\X^*(Q)$,
and there is a canonical isomorphism of Tate duality
\[H^1(\R,Q)\isoto \Hom\big( H^1(\Gal(\C/\R),\X^*(Q)\hs),\hs\Z/2\Z\hs\big);\]
see Milne \cite[Theorem I.2.13(b)]{Milne-ADT}.
\end{subsec}

\begin{subsec}
We say that a connected $\R$-group $G$ is \emph{compact}
if its group of real points is compact,
that is, if $G$ is reductive and anisotropic.
In particular, an $\R$-torus is compact if and only if it is anisotropic.
\end{subsec}

\begin{subsec}
From now on, $G$ is a quasi-connected reductive $\R$-group.
We wish to compute the Galois cohomology set $H^1(\R,G)$.
We denote $H=G^\sss$, which is a connected semisimple $\R$-group.
Let $T_0$ be a maximal compact torus in $H$.
Set $T =\cZ_{H}(T_0)$,
where $\cZ_{H}$ denotes the centralizer  in $H$.
Then $T $ is a maximal torus in $H$,
defined over $\R$; see \cite[Section 7]{Borovoi-arXiv}.
Set $Q=\cZ_G(T_0)$.
By Definition \ref{d:qc} we have $G=H\cdot Z(G)$
(that is, $G(\C)=H(\C)\cdot Z(G)(\C)$\hs).
It follows that $Q=\cZ_H(T_0)\cdot Z(G)=T \cdot Z(G) $.
Then $Q$ is an $\R$-quasi-torus in $G$ containing $T_0$,
and it is a maximal quasi-torus in $G$.
We say that $Q$ is a {\em fundamental quasi-torus} in $G$.

We set $N=\cN_G(T)$ and  $N_0=\cN_G(T_0)$,
where  $\cN_G$ denotes the normalizer in $G$.
Since $G=H\cdot Z(G)$, we have $N=\cN_{H}(T)\cdot Z(G)$
and  $N_0=\cN_{H}(T_0)\cdot Z(G)$.
We set $W=N/Q=\cN_{H}(T) /T$ and   $W_0=N_0/Q=\cN_{H}(T_0)/T$;
then $W$ is the Weyl group of $H$ with respect to the maximal torus $T$.
Since $T =\cZ_{H}(T_0)$, we have
$\cN_{H}(T_0)\subseteq \cN_{H}(T)$ and hence $W_0\subseteq W$.

We have a right action of $W_0$ on $T_0$ defined by
$(t,w)\mapsto t^w\coloneqq n^{-1}t\hs n$, where $t\in T_0$, $n\in N_0$,
and $n$ represents $w\in W_0$. This action is defined over $\R$.
We show that $W_0$ acts on $T_0$ effectively.
Indeed, if $w\in W_0$ with a representative $n\in N_0$ acts
trivially on $T_0$, then $n^{-1}t\hs n=t$ for any $t\in T_0$,
hence $n\in Q$ (because the centralizer of $T_0$ in $G$ is $Q$),
and therefore $w=1$.

Note that $N_0$ normalizes both $Q=\cZ_G(T_0)$ and $T=\cZ_H(T_0)$,
and hence $W_0$ acts on $Q$ and on $T$ by conjugation.
\end{subsec}

\begin{lemma} \label{l:W0}
$W_0(\C)=W_0(\R)=W(\R)$.
\end{lemma}

\begin{proof}
The group $W_0$  acts on $T_0$ effectively
and hence embeds into $\Aut_\C(T_0)$.
Since  $T_0$ is a compact torus,
all complex automorphisms of $T_0$
are defined over $\R$.
We see that the complex conjugation  acts trivially on $\Aut_\C(T_0)$,
and hence on $W_0$.
Thus  indeed $W_0(\C)=W_0(\R)$.

Since $W_0\subseteq W$, we have $W_0(\R)\subseteq W(\R)$.
Conversely, let $w\in W(\R)$.
Then $w$, when acting by conjugation on $T_\C$, preserves
the real structure of $T$, and hence it preserves
the maximal compact subtorus $T_0$ of $T$, that is $w\in W_0(\R)$.
Thus $W(\R)=W_0(\R)$, which completes the proof of the lemma.
\end{proof}

\begin{remark}
The group $W_0$ is isomorphic to the Weyl group $W(\ov R)$
of the  restricted root system $\ov R=R(H_\C,T_{0\C})$ (not necessarily reduced);
see \cite[Proposition 7.11(iii)]{BT}.
If $H$ is an  inner form of a compact $\R$-group, then $T_0=T$ and $W_0=W$.
If $H$ is an {\em outer} form of a {\em simple} compact $\R$-group,
then the type of the root system $\ov R$ is given in the table
in Gorbatsevich, Onishchik, and Vinberg  \cite[Section 3.3.9, page 119]{GOV}.
For example, if $R$ is of the type $\DDD_\ell$  ($\ell\ge 4$), then $\ov R$ is of the type $\BBB_{\ell-1}$,
and if $R$ is of the type $\EEE_6$, then $\ov R$ is of the type $\FFF_4$.
\end{remark}

\begin{construction}
We define a right action of $W_0(\R)=W_0(\C)$ on $H^1(\R,Q)$.
Let $c\in Z^1(\R,Q)$, $c$ represents $\xi\in H^1(\R,Q)$.
Let  $n\in N_0$, $n$ represents $w\in W_0$.
We set
\begin{equation*}
\xi * w=\Cl(n^{-1}c\hs\bar{n}),
\end{equation*}
where $\Cl$ denotes the cohomology class.

We show that $*$ is a well-defined action.
  First, since  $c\in Q$ and $N_0$ normalizes $Q$, we see
that $n^{-1}c\hs n\in Q$.
  Now $w\in W_0(\R)$, whence $w^{-1}\hs\overline{w}=1$ and
$n^{-1}\bar{n}\in Q$.
  It follows that $n^{-1}c\hs \bar{n}=n^{-1}c\hs n\cdot n^{-1}\bar{n}\in Q$.
We have
\begin{equation*}
n^{-1}c\hs\bar{n}\cdot \overline{n^{-1}c\hs\bar{n}}=
   n^{-1}c\hs\bar{n}\hs\bar{n}^{-1}\bar{c}\hs n
   =n^{-1}c\hs\bar{c}\hs n=n^{-1}n=1,
\end{equation*}
because $c\hs\bar{c}=1$.
  Thus $n^{-1}c\hs \bar{n}\in Z^1(\R, Q)$.
   If $c'\in Z^1(\R,Q)$ is another representative of $\xi$,
then $c'=q^{-1}c\hs \bar{q}$ for some $q\in Q$, and
\begin{equation*}
n^{-1}c'\hs\bar{n}=n^{-1}q^{-1}c\hs\bar{q}\hs\bar{n}
      =(n^{-1}q\hs n)^{-1}\cdot n^{-1}c\hs\bar{n}\cdot \overline{n^{-1}q\hs n}
      =(q')^{-1}\cdot n^{-1}c\hs\bar{n}\cdot\overline{q'},
\end{equation*}
where $q'=n^{-1}q\hs n\in Q$.
   We see that the cocycle $n^{-1}c'\hs\bar{n}\in Z^1(\R,Q)$
is cohomologous to $n^{-1}c\hs\bar{n}$.
   If $n'$ is another representative of $w$ in $N_0$, then
$n'=n\hs q$ for some $q\in Q$, and
  $(n')^{-1}c\hs\overline{n'}=q^{-1}\cdot n^{-1}c\hs\bar{n}\cdot\bar{q}$.
  We see that $(n')^{-1}c\hs\overline{n'}$ is cohomologous to
$n^{-1}c\hs\bar{n}$.
Thus $*$ is indeed a well-defined action
of the finite group $W_0$ on the set $H^1(\R,Q)$.

Note that in general $1*w=\Cl(n^{-1}\bar{n})\neq 1$, and therefore
the action $*$  does not respect the group structure in $H^1(\R,Q)$.

Let $\xi\in H^1(\R,Q)$ and $w\in W_0$.
It follows from the definition of the action\, $*$\, that the images of
$\xi$ and $\xi * w$ in $H^1(\R,G)$ are equal.
We see that the map $H^1(\R,Q)\to H^1(\R,G)$ induces a map
$H^1(\R,Q)/W_0\to H^1(\R,G)$.
\end{construction}

\begin{theorem}\label{thm:bijectivity}
  Let $G$, $H$, $T_0 $, $Q$, and $W_0$ be as above,
in particular, $Q$ be a fundamental quasi-torus
in a quasi-connected reductive $\R$-group $G$.
Then the map
$$H^1(\R,Q)/W_0\to H^1(\R,G)$$
induced by the map $H^1(\R,Q)\to H^1(\R,G)$ is  bijective.
\end{theorem}

\begin{proof}
We prove the surjectivity.
Since $G$ is quasi-connected reductive, we have
$G=H\cdot Z(G) $.
It follows that $Z(H) =H\cap Z(G) $ and that the
inclusion homomorphism $H\into G$ induces an isomorphism $H/Z(H) \isoto G/Z(G) $.
We consider  the surjective homomorphism
$\Ad\colon G \to G/Z(G) $
and its restriction to $H$
$$\Ad_H\colon\hs H\,\to\, G/Z(G) =H/Z(H) ,$$
which is surjective as well.
We write $H^\ad=H/Z(H) $, which is a semisimple $\R$-group of adjoint type.

Let $\Tt(H)$ and $\Tt(H^\ad)$ denote the sets of all $\R$-tori
in $H$ and $H^\ad$, respectively.
Consider the mutually inverse, inclusion-preserving bijective maps
\[\Tt(H)\to \Tt(H^\ad),\quad S\mapsto S_a\coloneqq\Ad_H(S)\
    \text{ for }S\in\Tt(H)\]
and
\[\Tt(H^\ad)\to \Tt(H),\quad S_a\mapsto
S\coloneqq(\hs\Ad_H^{-1}(S_a)\hs)^0 \ \text{ for }S_a\in\Tt(H^\ad).\]
The tori $S\subseteq H$ and
$\Ad_H(S)\subseteq H^\ad$ are isogenous,
and therefore, if one of them is compact (anisotropic), then so is the other.
We conclude that $S\subseteq H$ is a maximal compact torus
in $H$ if and only if
$\Ad_H(S)$ is a maximal compact torus in $H^\ad$.

Set $T_0^\ad=\Ad_H(T_0)$,
which is a maximal compact torus in $H^\ad$.
Set  $T^\ad =\cZ_{H^\ad}(T_0^\ad)$.
By \cite[Section 7]{Borovoi-arXiv}, $T^\ad$
is a maximal  torus in $H^\ad$,
and hence a fundamental torus
(a maximal torus containing a maximal compact torus).
The inverse image  $\Ad_H^{-1}(T^\ad)\subseteq H$
is a maximal torus containing $T_0$
and hence coincides with $T$.

We show that $\Ad^{-1}(T^\ad)=Q$.
Since $Q=T\cdot Z(G) $, we have $\Ad(Q)\subseteq T^\ad$
and hence $Q\subseteq\Ad^{-1}(T^\ad)$.
Conversely, let $y\in\Ad^{-1}(T^\ad)$.
We write $y=s\cdot z$ with $s\in H, z\in Z(G) $.
We have $\Ad(s)=\Ad(y)\in T^\ad$, whence $s\in \Ad_H^{-1}(T^\ad)=T$
and $y\in T\cdot Z(G) =Q$, as required.

Let $c\in Z^1(\R,G)$; then $\Ad(c)\in Z^1(\R,H^\ad)$.
By Kottwitz \cite[Lemma 10.2]{Ko}, or \cite[Theorem 1]{Bo},
or \cite[Theorem 9]{Borovoi-arXiv},
the natural map $H^1(\R,T^\ad)\to H^1(\R,H^\ad)$ is surjective.
Therefore, the cocycle $\Ad(c)$ is cohomologous to a cocycle in $T^\ad$,
that is, $a^{-1}\Ad(c)\hs\bar a\in T^\ad$ for some $a\in H^\ad$.
We lift the element  $a$ to an element $g\in G$. Then
$$\Ad(g^{-1}c\hs\bar g)=a^{-1}\Ad(c)\hs\bar a\in T^\ad,$$
and therefore $g^{-1}c\hs\bar g\in \Ad^{-1}(T^\ad)=Q$, that is,
$c$ is cohomologous to a cocycle in $Q$.
This proves the surjectivity in Theorem \ref{thm:bijectivity}.

We prove the injectivity in  Theorem \ref{thm:bijectivity}.
Let $c,c'\in Q$ be two cocycles, that is,
$c\hs\bar{c}=1$, $c'\,\ov{c'}=1$,
and assume that $c=x^{-1}c'\hs\bar{x}$, where $x\in G$.
We shall show that $c=n^{-1}c'\hs\bar{n}$ for some $n\in N_0$.

Consider the antilinear automorphism $\nu$ of $G$ that sends $g\in G$ to
$\nu(g)=c\hs\bar{g}\hs c^{-1}$.
Since $c$ is a cocycle, we have $\nu^2=\id$,
and in this way we obtain a twisted form $_c G$ of $G$.
 Clearly, we have $\nu(H)=H$, and thus
$\nu$ defines a twisted form $_c H$ of $H$.
 Since $c\in Q=\cZ_G(T_0)=\cZ_G(T)$, the embeddings
 of  $T_0$ and of $T$
into ${}_c H$ are defined over $\R$.
 We denote the corresponding $\R$-subgroups of $_c H$
again by $T_0$ and $T$, respectively.
 The compact torus $T_0$ of $_c H$ is contained
in some maximal compact torus $S$ of $_c H$,
and clearly $S$ is contained in the centralizer
$T$ of $T_0$ in $_c H$.
 Since $T_0$ is the largest compact subtorus of $T$,
we conclude that $S=T_0$.
 Thus $T_0 $ is a maximal compact torus in $_c H$.

Consider the embedding
\[i_x\colon\hs  T_0 \hs\into({}_c G)_\C,\quad
     t\mapsto x^{-1} t\hs x\ \text{ for }t\in T_0.\]
Then $i_x(T_0)\subseteq H$.
We have $\nu(\hs i_x(t)\hs)=c\hs\bar{x}^{-1}\bar{t}\hs\bar{x}\hs c^{-1}$.
Since $c\hs \bar{x}^{-1}=x^{-1}c'$, we obtain
\begin{equation*}
c\hs\bar{x}^{-1}\bar{t}\hs\bar{x}\hs c^{-1}=x^{-1}c'\hs\bar{t}\hs c^{\prime\,-1}x
=x^{-1}\bar{t}\hs x =i_x(\bar{t}\hs),
\end{equation*}
because $c'\in Q=\cZ_G(T_0)$.
We see that $\nu(\hs i_x(t)\hs)=i_x(\bar{t}\hs)$, and
hence the embedding $i_x$ is defined over $\R$.
  Set $T_0'=i_x(T_0)$.
We have $T_0'= x^{-1} T_0\hs x$.
The $\R$-torus $T_0'\subseteq {}_c H$ is compact,
and $\dim T_0'=\dim T_0$.
By \cite[Lemma 6]{Borovoi-arXiv},  $T_0'$
is conjugate to $T_0$ under ${}_c H(\R)$,
that is, $T_0 =h^{-1}\hs T_0' \hs h$
for some $h\in{}_c H(\R)$.
Set $n=xh$. Then
$$n^{-1}T_0 \,n=h^{-1}x^{-1}T_0\hs  x\hs h
     =h^{-1}\hs T_0' \,h=T_0 \hs,$$
whence $n\in \cN_{G}(T_0)=N_0$.
  The condition $h\in{}_c H
(\R)$ means that $c\hs\bar{h}\hs c^{-1}=h$, or
$h^{-1}c\hs\bar{h}=c$.
  It follows that
\begin{equation*}
n^{-1}c'\hs\bar{n}=h^{-1}x^{-1}c'\hs\bar{x}\hs\bar{h}=h^{-1}c\hs\bar{h}=c.
\end{equation*}
We have showed that there exists $n\in N_0$ such that
$c=n^{-1}c'\hs\bar{n}$, and hence  the cohomology classes
$\Cl(c), \Cl(c')\in H^1(\R,Q)$ lie in the same orbit of $W_0$.
This proves the injectivity and completes
the proof of Theorem \ref{thm:bijectivity}.
\end{proof}

\begin{remark}
If $G$ is a {\em connected, compact} $\R$-group,
then Theorem \ref{thm:bijectivity}
says that $H^1(\R,G)=T(\R)_2/W$, where $T$ is a maximal torus in $G$,
$T(\R)_2$ is the group of elements of order dividing 2 in $T(\R)$,
and $W=W(G_\C\hs, T_\C)$ is the Weyl group with the usual conjugation action on $T$.
This was earlier proved by Borel and Serre \cite{BS}.
If $G$ is a {\em connected} reductive $\R$-group, not necessarily compact,
then Theorem \ref{thm:bijectivity} says that $H^1(\R,G)=H^1(\R,T)/W_0$,
where $T$ is a fundamental torus in $G$.
This was earlier proved in \cite{Bo}.
\end{remark}

\subsection*{Acknowledgements}
The authors are very grateful to N.\,Q. Th\v{a}\'ng for the reference to Labesse \cite{La}.
We thank B.\,\`E. Kunyavski\u{\i} and  J.\,S. Milne for  helpful comments.


\begin{thebibliography}{99}

\bibitem{Be} G. Berhuy, {\em An Introduction to Galois Cohomology and its Applications.}
 Cambridge University Press, Cambridge, 2010.

\bibitem{BS} A. Borel et J.-P. Serre, {\em Th\'eor\`emes de finitude
en cohomologie galoisienne.} Comm. Math. Helv. {\bf 39} (1964), 111--164
(= A. Borel, {\OE}uvres: Collected papers, 64, Vol. II, Springer-Verlag, Berlin, 1983).

\bibitem{Bo}
M.\,V. Borovoi, {\em Galois cohomology of real reductive groups,
and real forms of simple Lie algebras.} Functional. Anal. Appl.
{\bf 22}:2 (1988), 135--136.

\bibitem{Bo-Memoir}
M. Borovoi, {\em  Abelian Galois cohomology of reductive groups.}
Mem. Amer. Math. Soc. {\bf 132} (1998), no. 626.

\bibitem{Borovoi-arXiv}
M. Borovoi,
{\em Galois cohomology of reductive algebraic groups over the field of real numbers.}
{\tt arXiv:1401.5913 [math.GR]}.

\bibitem{BE}
M. Borovoi and Z. Evenor,
{\em Real homogenous spaces, Galois cohomology, and Reeder puzzles.}
J. Algebra {\bf 467} (2016), 307--365.

\bibitem{BT}
M. Borovoi and D.\,A. Timashev,
{\em Galois cohomology of real semisimple groups via Kac labelings.}
Transform. Groups {\bf 26} (2021), 433--477.

\bibitem{GOV}
V.\,V.~Gorbatsevich, A.\,L.~Onishchik, and  E.\,B. Vinberg,
{\em Structure of Lie groups and Lie algebras.}
{Lie Groups and Lie Algebras III,}
Encyclopaedia of Mathematical Sciences, Vol. 41,
Springer-Verlag,  Berlin, 1994.

\bibitem{Hu} J.\,E. Humphreys, {\em Linear Algebraic Groups.}
Springer-Verlag, Berlin, 1975.

\bibitem{Ko} R.\,E. Kottwitz,
{\em Stable trace formula: elliptic singular terms.}
 Math. Ann. {\bf 275} (1986), 365--399.

\bibitem{La}
J.-P. Labesse,
{\em Cohomologie, stabilisation et changement de base.}
Ast\'erisque {\bf 257} (1999).

\bibitem
{Milne-ADT}
J.\,S. Milne,
{\em Arithmetic Duality Theorems.}
Second edition, BookSurge, LLC, Charleston, SC, 2006.

\bibitem
{Milne-AG}
J.\,S. Milne,
{\em Algebraic groups.
The theory of group schemes of finite type over a field.}
Cambridge Studies in Advanced Mathematics, 170,
Cambridge University Press, Cambridge, 2017.

\bibitem
{NP}
A. Nair and D. Prasad,
{\em Cohomological representations for real reductive groups.}
{\tt arXiv:1904.00694 [math.RT]}.

\bibitem{S} J.-P. Serre, {\em Galois Cohomology.}
Springer-Verlag, Berlin, 1997.

\end{thebibliography}
\end{document}